\documentclass[12pt]{amsart}

\usepackage{multido}
\usepackage{graphicx}
\usepackage{amssymb}
\usepackage{mathabx}
\usepackage{hyperref}

\numberwithin{equation}{section}
\numberwithin{figure}{section}

\allowdisplaybreaks
\DeclareSymbolFont{bbold}{U}{bbold}{m}{n}
\DeclareSymbolFontAlphabet{\mathbbold}{bbold}

\theoremstyle{plain} \newtheorem{theorem}{Theorem}[section]
\theoremstyle{plain} \newtheorem{proposition}[theorem]{Proposition}
\theoremstyle{plain} \newtheorem{lemma}[theorem]{Lemma}
\theoremstyle{plain} 
\theoremstyle{definition} \newtheorem{definition}[theorem]{Definition}
\theoremstyle{definition} \newtheorem{notation}[theorem]{Notation}
\theoremstyle{remark} \newtheorem{remark}[theorem]{Remark}
\theoremstyle{remark} \newtheorem{example}[theorem]{Example}

\theoremstyle{definition} \newtheorem{assumption}{Assumption}

\newcommand{\bD}{\mathbb{D}}
\newcommand{\bE}{\mathbb{E}}

\newcommand{\bK}{\mathbb{K}}

\newcommand{\bN}{\mathbb{N}}
\newcommand{\bP}{\mathbb{P}}
\newcommand{\bQ}{\mathbb{Q}}
\newcommand{\bR}{\mathbb{R}}
\newcommand{\bS}{\mathbb{S}}

\newcommand{\bX}{\mathbb{X}}

\newcommand{\cB}{\mathcal{B}}

\newcommand{\cH}{\mathcal{H}}
\newcommand{\cK}{\mathcal{K}}

\renewcommand{\AA}{\mathbf{A}}
\newcommand{\one}{\mathbf{1}}

\newcommand{\bKH}{\hat{\bK}}
\newcommand{\bKD}{\tilde{\bK}}

\newcommand{\neutral}{\mathbf{e}}

\renewcommand{\epsilon}{\varepsilon}
\newcommand{\eps}{\epsilon}
\newcommand{\eqd}{\overset{d}{\sim}}
\renewcommand{\phi}{\varphi}
\renewcommand{\Re}{\mathrm{Re}\,}
\newcommand{\salg}{\mathfrak{B}}

\usepackage{fancybox}
\newlength{\querylen}
\begin{document}

\title[Polar decomposition of scale-homogeneous measures]{Polar decomposition of
  scale-homogeneous measures with application to L\'evy measures of strictly
  stable laws}

\author[S.N. Evans]{Steven N. Evans}
\address{Department of Statistics \#3860\\
 367 Evans Hall \\
 University of California \\
  Berkeley, CA  94720-3860 \\
   USA} 
\email{evans@stat.berkeley.edu}

\author[I. Molchanov]{Ilya Molchanov}
\address{University of Bern \\
 Institute of Mathematical Statistics and Actuarial Science \\
 Sidlerstrasse 5 \\
 CH-3012 Bern \\
 SWITZERLAND}
\email{ilya.molchanov@stat.unibe.ch}

\thanks{SNE supported in part by NSF grant DMS-09-07630 and NIH grant
  1R01GM109454-01. IM supported in part by Swiss National Science
  Foundation grants 200021-137527 and 200021-153597.}



\date{\today}

\begin{abstract}
  A scaling on some space is a measurable action of the group of
  positive real numbers.  A measure on a measurable space equipped
  with a scaling is said to be $\alpha$-homogeneous for some nonzero
  real number $\alpha$ if the mass of any measurable set scaled by any
  factor $t > 0$ is the multiple $t^{-\alpha}$ of the set's original
  mass.  It is shown rather generally that given an
  $\alpha$-homogeneous measure on a measurable space there is a
  measurable bijection between the space and the Cartesian product of
  a subset of the space and the positive real numbers (that is, a
  ``system of polar coordinates'') such that the push-forward of the
  $\alpha$-homogeneous measure by this bijection is the product of a
  probability measure on the first component (that is, on the
  ``angular'' component) and an $\alpha$-homogeneous measure on the
  positive half-line (that is, on the ``radial'' component). This
  result is applied to the intensity measures of Poisson processes
  that arise in L\'evy-Khinchin-It\^o-like representations of
  infinitely divisible random elements. It is established that if a
  strictly stable random element in a convex cone admits a series
  representation as the sum of points of a Poisson process, then it
  necessarily has a LePage representation as the sum of i.i.d. random
  elements of the cone scaled by the successive points of an
  independent unit intensity Poisson process on the positive half-line
  each raised to the power $-\frac{1}{\alpha}$.
\end{abstract}

\maketitle

\section{Introduction}
\label{sec:introduction}

One may consider infinitely divisible random elements in very general
settings where there is a binary operation on the carrier space such
that it makes sense to speak of the ``addition'' of two random
elements.  Once one rules out pathologies such as random elements with
idempotent distributions (that is, probability measures that are equal
to their convolution with themselves), the main step towards
understanding infinitely divisible random elements usually consists of
establishing their representation as a sum of the points of a Poisson
random measure and possibly also constant terms and random elements
with Gaussian-like distributions. Results of this type start from the
classical L\'evy-Khinchin theorem for Euclidean space, with further
extensions to Banach spaces \cite{arauj:gin80,ros90}, groups
\cite{par67}, function spaces \cite{bas:ros13,kab:stoev12,raj:ros89},
and general semigroups \cite{MR974112}.

A {\em convex cone} is a semigroup equipped with a scaling, that is,
with an action of the group of positive real numbers that interacts
suitably with the semigroup operation.  There is a natural notion of
stability for random elements of such a carrier space that extends the
usual notion of stability for random elements of Euclidean space, see
\cite{dav:mol:zuy08}.  For strictly stable random elements, the
intensity measure of the corresponding Poisson process -- the L\'evy
measure of the random element -- is scale-homogeneous, that is,\ its value
on any set scaled by any positive real number equals the value on the
original set up to a factor given by a (fixed) power of the scaling
constant.  A scale-homogeneous measure on $\bR^d$ that is finite
outside a neighborhood of the origin can be represented in polar
coordinates as the product of a probability measure on the unit sphere
(the directional or angular component) and a scale-homogeneous measure
on the positive reals (the radial component). A nontrivial
scale-homogeneous measures on the positive reals that is finite
outside neighborhoods of the origin is necessarily of the form $c \,
\alpha \, t^{-(\alpha + 1)} \, dt$ for some constants $c > 0$ and
$\alpha > 0$.  The unit sphere need not be the unit sphere for the
Euclidean norm: the unit sphere for any norm on $\bR^d$ serves equally
well.  The situation becomes more complicated for more general spaces
where there may be no natural candidates to play the r\^ole of the
unit sphere.

Such a polar decomposition of the L\'evy measure leads to a
representation of stable random elements as the sum of a series
obtained from a sequence of i.i.d. random elements scaled by the
points of a Poisson process on the positive real line with intensity
measure of the form $\alpha \, t^{-(\alpha + 1)} \, dt$ or,
equivalently by points of the form $\Gamma_k^{-\frac{1}{\alpha}}$,
where $(\Gamma_k)_{k \in \bN}$ are the points of a unit intensity
Poisson process on the positive real line. Such a representation was
first obtained in \cite{lep:wood:zin81} for Euclidean space and since
then it and its various extensions have been called the {\em LePage
  series} representation. For example, such a representation was given
for random elements of a suitable space of functions with the
semigroup operation being addition in \cite{sam:taq94} and with the
semigroup operation being maximum in \cite{haan84}.

Section~\ref{sec:decomp-homog-meas} presents the main result that
concerns a polar decomposition of scale-homogeneous measures on quite
general spaces. Proposition~\ref{proposition:renorm} is inspired by
the argument used in \cite[Th.~10.3]{evan:mol15} to derive the LePage
series representation for stable metric measure spaces. A similar
rescaling argument was used in the context of max-stable sequences in
\cite{haan84}.

In Section~\ref{sec:strictly-stable-rand} we review the definition of
a convex cone, introduce stable random elements on such carrier
spaces, and apply the general decomposition results to derive the
series representation of strictly stable random elements in convex
cones. Such a series representation is a consequence of the
polar decomposition of the L\'evy measure and expresses any strictly
stable random element as a sequence of i.i.d. random elements scaled
by a power of the successive points of the unit intensity Poisson
point process on the positive half line.  The key hypotheses we require are a
one-sided continuity property of bounded semicharacters and the existence of a
measurable transversal for the action of the positive real numbers on
the carrier space. 

Section~\ref{sec:applications} presents several applications, some in
well-known settings and others in more novel contexts, where our
approach streamlines the derivation of the LePage
series representation of a stable random elements.

\section{Decomposition of scale-homogeneous measures}
\label{sec:decomp-homog-meas}

\begin{definition}
  Let $\bX$ be a measurable space with a $\sigma$-algebra $\salg$, and
  let $(x,t)\mapsto tx\in\bX$ be a \emph{pairing} of $x\in\bX$ and
  $t\in\bR_{++}:=(0,\infty)$ which is $\salg \otimes \cB(\bR_{++}) /
  \salg$-measurable, where $\cB(\bR_{++})$ is the Borel
  $\sigma$-algebra on $\bR_{++}$. Assume that this pairing is an {\em
    action} of the group $\bR_{++}$ (where the group operation is the
  usual multiplication of real numbers) on $\bX$; that is, for all
  $t,s>0$ and $x\in\bX$, $t(sx)=(ts)x$ and $1 x = x$.  We refer to
  such a pairing as a \emph{scaling}.
\end{definition} 

\begin{remark}
  \label{remark:scaling-one-s}
  For $B\in\salg$ and $t\in\bR_{++}$, set $tB:=\{tx:\; x\in B\}$.  The
  measurability of the map $x\mapsto tx$ for each $t\in\bR_{++}$
  yields that $s B=\{x:\; (s^{-1})x\in B\}\in\salg$ for all
  $B\in\salg$ and $s\in\bR_{++}$.
\end{remark}  

\begin{definition}
  A measure $\nu$ on $\bX$ is said to be \emph{$\alpha$-homogeneous}
  for some $\alpha\in\bR \setminus \{0\}$ if
  \begin{equation}
    \label{eq:nu-x-homogeneous}
    \nu(sB)=s^{-\alpha}\nu(B),\qquad B\in\salg,\; s>0. 
  \end{equation}
\end{definition}

\begin{remark}
  By redefining the pairing to be $(x,t) \mapsto t^{-1} x$, we may
  assume that $\alpha > 0$ and we will do so from now on.  Moreover,
  by redefining the pairing to be $(x,t) \mapsto t^{\frac{1}{\alpha}}
  x$ we could even assume that $\alpha = 1$, but we choose not to do
  so in order that the notation $t x$ retains its usual meaning when
  $\bX$ is a subset of some vector space over $\bR$.
\end{remark}

We begin with a technical lemma.

\begin{lemma}
  \label{lemma:extend_homo_completion}
  Suppose that $\nu$ is an $\alpha$-homogeneous measure on $(\bX,
  \salg)$.  Write $\salg^\nu$ for the completion of the
  $\sigma$-algebra $\salg$ with respect to the measure $\nu$.  Then $s
  B \in \salg^\nu$ with $\nu(sB)=s^{-\alpha}\nu(B)$ for all $s > 0$
  and $B \in \salg^\nu$.
\end{lemma}
\begin{proof}
  Suppose that $A \subset \bX$ is a $\nu$-null set.  That is, there
  exists $N \in \salg$ such that $A \subseteq N$ and $\nu(N) = 0$.
  For any $s > 0$, $s A \subseteq s N \in \salg$ and $\nu(s N) =
  s^{-\alpha}\nu(N) = 0$ by \eqref{eq:nu-x-homogeneous}, and so $s A$
  is also $\nu$-null.  Now, $B \in \salg^\nu$ if and only if there
  exists $C \in \salg$ such that $B \triangle C$ is $\nu$-null, in
  which case $\nu(B) = \nu(C)$.  Since $(s B) \triangle (s C) = s (B
  \triangle C)$ for $s > 0$ and the latter set is $\nu$-null by the
  above, we see that $s B \in \salg^\nu$ and $\nu(s B)= \nu(s C) =
  s^{-\alpha}\nu(C) = s^{-\alpha}\nu(B)$.
\end{proof}

\begin{remark}
  Note that Lemma~\ref{lemma:extend_homo_completion} implies that the
  map $x \mapsto s x$ is $\salg^\nu / \salg^\nu$-measurable for any
  $s>0$.  It does not say that the map $(x, s) \mapsto s x$ is
  $\salg^\nu \otimes \cB(\bR_{++}) / \salg^\nu$-measurable.
\end{remark}

\begin{notation}
  For $I\subseteq\bR_{++}$ and $B \in \salg$, put
  \begin{displaymath}
    IB:=\bigcup_{t\in I} tB. 
  \end{displaymath}
\end{notation}

The following is the first of a series of related assumptions that
require the existence of a suitably rich family of subsets or 
nonnegative functions on our carrier space.

\begin{assumption}
  \label{assumption:U}
  There exist sets $U_k\in\salg$, $k\in\bN$, such that
  \begin{itemize}
  \item[(i)] if $x\in U_k$, then $tx\in U_k$ for all $t\geq1$ 
  (that is, $[1,\infty) U_k = U_k$);
  \item[(ii)] the sets $V_k:=(0,\infty) U_k$, $k\in\bN$, cover $\bX$
  (that is, $\bigcup_{k \in \bN} V_k = \bX$). 
  \end{itemize}
\end{assumption}

\begin{proposition}
  \label{proposition:renorm}
  Suppose that Assumption~\ref{assumption:U} holds.
  The following are equivalent for a measure $\nu$ on $(\bX, \salg)$.
  \begin{itemize}
  \item[i)]
    The measure $\nu$ is a nontrivial $\alpha$-homogeneous measure with
    \begin{equation}
      \label{Eq:nu-finite}
      \nu(U_k)<\infty,\qquad k\in\bN\,.
    \end{equation}
  \item[ii)] The measure $\nu$ is the push-forward of the measure
    $\pi\otimes\theta_\alpha$ by the map $(x,t)\mapsto tx$, where
    $\pi$ is a probability measure on $\bX$ such that
    \begin{equation}
      \label{eq:pi-Sk}
      \int_0^\infty \pi(tU_k)t^{\alpha-1} \, dt < \infty,\qquad k\in\bN,
    \end{equation}
    and $\theta_\alpha$ is the measure on $\bR_{++}$ given by 
    $\theta_\alpha(dt) := \alpha t^{-(\alpha+1)} \, dt$.
  \item[(iii)] For a probability measure $\pi$ on $\bX$ such that
    \eqref{eq:pi-Sk} holds,
    \begin{equation}
      \label{Eq:nu-rep}
      \nu(B) = \alpha \int_0^\infty \pi(t B) \, t^{\alpha-1} \, dt
    \end{equation}
    for all $B \in \salg$.
  \end{itemize}
\end{proposition}
	
\begin{proof}
  Statement (iii) is just a restatement of statement (ii), so it
  suffices to show that (i) $\Longrightarrow$ (ii) and (iii)
  $\Longrightarrow$ (i).

  \textsl{(i) $\Longrightarrow$ (ii).}  Fix $k\in\bN$. Note that
  $V_k=\bigcup_{n\in\bN} 2^{-n} U_k$, so that $V_k\in\salg$.  Put
  \begin{displaymath}
    \bar U_k:=\bigcap_{t<1} t U_k
    = \bigcap_{n \in \bN} (1 - 2^{-n}) U_k \in\salg,
  \end{displaymath}
  and observe that $V_k=(0,\infty) \bar U_k$.  The $\alpha$-homogeneity of
  $\nu$ and \eqref{Eq:nu-finite} yield that
  \begin{displaymath}
    \nu(\bar U_k)= \inf_{t<1} t^{-\alpha}\nu(U_k)=\nu(U_k)<\infty.
  \end{displaymath}
  For $x\in V_k$, set
  \begin{displaymath}
    \tau(x):=\sup\{t>0:\; x\in t\bar U_k\} \in (0,\infty]. 
  \end{displaymath}
  Since
  \begin{displaymath}
    \{x\in V_k:\; \tau(x)\geq t\}=\bigcap_{s<t} s\bar U_k
    = t \bigcap_{s<1}s\bar U_k = t\bar U_k,
  \end{displaymath}
  the function $\tau: V_k\mapsto (0,\infty]$ is
  $\salg$-measurable. Clearly, $\tau(sx)=s\tau(x)$ for all $s>0$.
  Also, setting
  \begin{displaymath}
    N_k := \{x\in V_k:\; \tau(x)=\infty\} = \bigcap_{t > 0} t \bar U_k,
  \end{displaymath}
  we have $s N_k = N_k$ for all $s > 0$ and
  \begin{displaymath}
    \nu(N_k)
    =\inf_{t>0} t^{-\alpha}\nu(\bar U_k)=0. 
  \end{displaymath}
  Define 
  \begin{displaymath}
    W_k:=\{x\in V_k:\; \tau(x)=1\} 
    = \bar U_k \setminus \left(\bigcup_{t > 1} t \bar U_k\right)
    \subseteq V_k \setminus N_k.
  \end{displaymath}
  Note that
  \begin{equation}
    \label{barU_k_minus_N_k_is_product}
    \bar U_k \setminus N_k = [1,\infty) W_k
  \end{equation}
  and
  \begin{equation}
    \label{V_k_minus_N_k_is_product}
    V_k \setminus N_k = (0,\infty) W_k.
  \end{equation}
  
  Set $V'_1:=V_1$ and $V'_i:= (V_i\cap (V_1\cup\cdots\cup V_{i-1})^c)$
  for $i\geq 2$. Then put $V_j'' = V_j' \setminus N_j$, $W_j''
  :=W_j\cap V_j''$ and $\bar U_j'':=\bar U_j \cap V_j''$ for
  $j\in\bN$. The sets $\{V_k''\}_{k\in\bN}$ are disjoint, $\bX
  \setminus \bigcup_{k \in \bN} V_k'' \subseteq \bigcup_{k \in \bN}
  N_k$ is a $\nu$-null set, and $\bar U_k''=\{tx:\; t\geq1,\, x \in
  W_k''\}$ for all $k\in\bN$. Therefore, it is possible to assume
  without loss of generality that all of the sets $N_k$ are empty, the
  sets $\{\bar U_k\}_{k \in \bN}$ are pairwise disjoint, the sets
  $\{V_k\}_{k\in\bN}$ form a measurable partition of $\bX$, and that
  $\tau(x)$ is uniquely defined and finite for all $x \in \bX$.

  The sets $\{x\in V_k:\; \tau(x)=t\}=t W_k$ are disjoint for
  different $t>0$ and their union is $V_k$.  The map $x\mapsto
  (\tau(x)^{-1}x,\tau(x))$ is therefore a well-defined bimeasurable
  bijection between $V_k$ and $W_k\times\bR_{++}$.  Let $\tilde \nu$
  be the push-forward of $\nu$ by the map $x \mapsto (\tau(x)^{-1} x,
  \tau(x))$ and define a measure $\rho_k$ on $W_k$ by $\rho_k(A) =
  \tilde \nu(A \times [1,\infty))$. Note that $\rho_k$ is a finite
  measure, since
  \[
  \rho_k(W_k) 
  = \tilde \nu(W_k \times [1,\infty))=\nu(\bar U_k)<\infty
  \]
  by \eqref{Eq:nu-finite}.  The $\alpha$-homogeneity of $\nu$ is
  equivalent to the scaling property $\tilde \nu(A \times s B) =
  s^{-\alpha} \tilde \nu(A \times B)$ for $s>0$, measurable $A
  \subseteq W_k$, and Borel $B \subseteq \bR_{++}$.  Thus, if we let
  $\theta_\alpha$ be the measure on $\bR_{++}$ given by
  $\theta_\alpha(dt) = \alpha t^{-(\alpha+1)} \, dt$, then
  \[
  \begin{split}
    \tilde \nu(A \times [b,\infty)) 
    & = \tilde \nu(A \times b [1,\infty)) \\
    & = b^{-\alpha} \tilde \nu(A \times [1,\infty)) \\
    & = \rho_k(A) \times \theta_\alpha([b,\infty)) \\
  \end{split}
  \]
  for $A \subseteq W_k$.  The restriction of $\tilde \nu$
  to $W_k \times \bR_{++}$ is therefore $\rho_k \otimes \theta_\alpha$ and
  hence the restriction of $\nu$ to $V_k$ is the push-forward of
  $\rho_k \otimes \theta_\alpha$ by the map $(y,t) \mapsto ty$.

  We can think of $\rho_k$ as a measure on all of $V_k$.  For $c_k
  \in \bR_{++}$, let $\eta_k$ be the measure on $V_k$ that assigns
  all of its mass to the set $c_k W_k$ and is given by $\eta_k(A) =
  c_k^\alpha \rho_k(c_k^{-1} A)$.  We have
  \[
  \begin{split}
    (\eta_k \otimes \theta_\alpha)(\{(y,t) : ty \in B\})
    & = \int \eta_k(t^{-1} B)  \alpha t^{-(\alpha+1)} \, dt \\
    & = \int c_k^\alpha \rho_k(c_k^{-1} t^{-1} B) \alpha t^{-(\alpha+1)} \, dt \\
    & = \int \rho_k(s^{-1} B) \alpha s^{-(\alpha+1)} \, ds \\
    & =  (\rho_k \otimes \theta_\alpha)(\{(y,t) : ty \in B\}) \\
    & = \nu(B) \\
  \end{split}
  \]
  for measurable $B\subset V_k$, 
  and so $\eta_k$ is a finite measure with total mass 
  $c_k^\alpha \rho_k(W_k)$ that has the property that
  the push-forward of $\eta_k \otimes \theta_\alpha$ by
  the map $(y,t) \mapsto ty$ is the restriction of
  $\nu$ to $V_k$.
  
  We can regard $\eta_k$ as being a finite measure on all of $\bX$
  and, by choosing the constants $c_k$, $k \in \bN$, appropriately we
  can arrange for $\pi := \sum_{k \in \bN} \eta_k$ to be a probability
  measure such that \eqref{Eq:nu-rep} holds, and this also implies
  \eqref{eq:pi-Sk}.

  \textsl{(iii) $\Longrightarrow$ (i).}  It is easy to see that $\nu$
  given by \eqref{Eq:nu-rep} is $\alpha$-homogeneous, and
  \eqref{Eq:nu-finite} follows from \eqref{eq:pi-Sk} and
  \eqref{Eq:nu-rep}.
\end{proof}

\begin{remark}
  A key observation in the proof of
  Proposition~\ref{proposition:renorm} is that it is possible to
  define measurable sets $W_k$ such that
  \eqref{barU_k_minus_N_k_is_product} and
  \eqref{V_k_minus_N_k_is_product} hold, and if $x \in W_k$, then $t x
  \notin W_k$ for any $t \ne 1$.  We would like to reverse the
  construction in Proposition~\ref{proposition:renorm} by starting
  with a suitable collection of sets $\{W_k\}_{k\in\bN}$ with this
  last property and conclude that if we put $\bar U_k := [1,\infty)
  W_k$, then $\bar U_k \in \salg$ and if we construct sets from each
  of the $\bar U_k$ in the manner that the $W_k$ are constructed in
  the proof, then we recover the $W_k$.  There is, however, a slightly
  delicate point here: if $B \in \salg$, then it is not necessarily
  the case that $[1,\infty) B = \bigcup_{t \ge 1} t B = \{t x: t \ge
  1, \, x \in B\} \in \salg$.
\end{remark}

\begin{lemma}
  \label{lemma:projection}
  Let $\nu$ be a $\sigma$-finite $\alpha$-homogeneous measure on
  $(\bX, \salg)$ and suppose that $\salg$ is $\nu$-complete.
  \begin{itemize}
  \item[i)] For $B \in \salg$ and a Borel set $I \subseteq \bR_{++}$,
    the set $I B$ also belongs to $\salg$.
  \item[ii)] If $\nu([t,u) B) < \infty$ for $B \in \salg$ and $0 < t <
    u < \infty$, then $\nu([s,\infty) B) < \infty$ for all $s > 0$.
\end{itemize}
\end{lemma}
\begin{proof}
  For part (i), observe that
  \[
  \begin{split}
    I B    
    & =
    \{y \in \bX : \; y = t x, \; \text{for some} \; t \in I \; \text{and} \; x \in B\} \\  
    & =
    \{y \in \bX : \; t^{-1} y = x, \; \text{for some} \; t \in I \; \text{and} \; x \in B\} \\ 
    & =
    \{y \in \bX : \; u y = x, \; \text{for some} \; u \in I^{-1} \; \text{and} \; x \in B\} \\ 
    & =
    \{y \in \bX : \; u y \in B, \; \text{for some} \; u \in I^{-1}\} \\
    & =
    \Pi(\{(y,u) \in \bX \times I^{-1}: \; u y\in B\}), \\
  \end{split}
  \]
  where $\Pi: \bX \times \bR_{++} \to \bX$ is the projection map
  defined by $\Pi((x,t)) = x$.  The map $t \mapsto t^{-1}$ from
  $\bR_{++}$ to itself is a $\cB(\bR_{++})$-measurable bijection and
  it is its own inverse, so the set $I^{-1}$ is
  $\cB(\bR_{++})$-measurable.  It follows that the set $\{(y,u) \in
  \bX \times I^{-1} :\; u y\in B\}$ is $\salg \otimes
  \cB(\bR_{++})$-measurable.  The projection theorem (see,
  for example, \cite[Th.~12.3.4]{MR752692} or \cite[Section~III.44]{MR521810})
  and the $\nu$-completeness of $\salg$ yield that $I B \in \salg$.

  For part (ii), note that 
  \[
  \begin{split}
    [s,\infty) B 
    & = \left(\bigcup_{n \geq0} \frac{s}{t} \left(\frac{u}{t}\right)^n [t,u)\right) B \\
    & \subseteq \bigcup_{n \geq0} \frac{s}{t} \left(\frac{u}{t}\right)^n [t,u) B. \\
  \end{split}
  \]
  By the $\alpha$-homogeneity of $\nu$
  \[
  \nu\left(\frac{s}{t} \left(\frac{u}{t}\right)^n [t,u) B\right)
  =
  \left(\frac{s}{t}\right)^{-\alpha} \left(\frac{u}{t}\right)^{-n \alpha} \nu([t,u) B),
  \]
  and the result follows from the subadditivity of $\nu$ and the
  summability of the relevant geometric series.
\end{proof}

It is possible to obtain the conclusion of
Proposition~\ref{proposition:renorm} under another assumption.

\begin{assumption}
  \label{assumption:X}
  There exists $\bS\in\salg$ such that 
  \begin{displaymath}
    \bX=\{tx:\; t\in\bR_{++}, \, x\in\bS\},
  \end{displaymath}
  and $\{x,tx\}\subset\bS$ for some $x\in\bX$ and $t>0$ implies that
  $t=1$. 
\end{assumption}

\begin{definition}
  The {\em orbits} of the action of the group $\bR_{++}$ on $\bX$ are
  the sets of the form $\{t x : t \in \bR_{++}\}$ for some $x \in
  \bX$.  The orbits are the equivalence classes for the equivalence
  relation on $\bX$ where we declare $x$ and $y$ to be equivalent if
  $y = t x$ for some $t \in \bR_{++}$.  Assumption~\ref{assumption:X}
  says that the measurable set $\bS$ intersects each equivalence class
  in exactly one point. Such a set is called a \emph{transversal}.
\end{definition}

The following result is immediate from Lemma~\ref{lemma:projection} and the
proof of Proposition~\ref{proposition:renorm}.

\begin{proposition}
  \label{proposition:sphere}
  Let $\nu$ be a $\sigma$-finite, $\alpha$-homogeneous measure on $\bX$.
  \begin{itemize}
  \item[i)] Suppose that Assumption~\ref{assumption:U} and the
    finiteness condition \eqref{Eq:nu-finite} hold.  Set
    \[
    W_k 
    := \left(\bigcap_{s<1} s U_k\right) 
   \setminus \left(\bigcup_{t > 1} t \left(\bigcap_{s<1} s U_k\right)\right).
    \]
    Then, by possibly replacing $\bX$ with a set $\bX' \in \salg$ such
    that $t \bX' = \bX'$ for all $t >0$ and $\nu(\bX \setminus
    \bX')=0$, Assumption~\ref{assumption:X} holds for $\bS :=
    \bigcup_{k \in \bN} W_k$.  Moreover, for any nonempty intervals
    $I_k := [a_k, b_k) \subset \bR_{++}$, $k \in \bN$, it is the
    case that $I_k W_k \in \salg$ and $\nu(I_k W_k)<\infty$ for
    $k\in\bN$.
  \item[ii)] Suppose that that Assumption~\ref{assumption:X} holds and
    there is a pairwise disjoint family of measurable sets
    $\{W_k\}_{k\in\bN}$ such that $\bS = \bigcup_{k \in \bN} W_k$ and
    a family of nonempty intervals $I_k := [a_k, b_k) \subset
    \bR_{++}$, $k \in \bN$, such that $\nu(I_k W_k)<\infty$ for
    all $k\in\bN$, where $I_k W_k$ is guaranteed to belong to
    $\salg^\nu$ for all $k \in \bN$.  Then
    Assumption~\ref{assumption:U} and the finiteness condition
    \eqref{Eq:nu-finite} hold with $(\bX,\salg)$ replaced by $(\bX,
    \salg^\nu)$ and $U_k := [1,\infty)W_k$ for $k\in\bN$.
  \end{itemize}
\end{proposition}

\begin{definition}
  Recall that a {\em Borel space} is a measurable space that is Borel
  isomorphic to a Borel subset of a Polish space equipped with the
  trace of the Borel $\sigma$-algebra on the Polish space.
\end{definition}

\begin{lemma}
  \label{lemma:pair-of-maps}
  Suppose that $\nu$ is a $\sigma$-finite $\alpha$-homogeneous measure
  on $(\bX,\salg)$ and that Assumption~\ref{assumption:X} is
  satisfied.  For $x \in \bX$, define $\tau(x) \in \bR_{++}$ by the
  requirement that $\tau(x)^{-1}x\in\bS$.  Then the following hold.
  \begin{itemize}
  \item[i)] The maps $(x,t)\mapsto tx$ from $\bS\times\bR_{++}$ to
    $\bX$ and $x\mapsto(\tau(x)^{-1}x,\tau(x))$ from $\bX$ to
    $\bS\times\bR_{++}$ are mutually inverse.
  \item[ii)] If $\salg$ is $\nu$-complete, then
    $x\mapsto(\tau(x)^{-1}x,\tau(x))$ is $\salg / \salg_{| \bS}
    \otimes \cB(\bR_{++})$-measurable, where $\salg_{| \bS}$ is the
    $\sigma$-algebra induced by $\salg$ on $\bS$.
  \item[iii)] If $(\bX,\salg)$ is a Borel space, then
    $x\mapsto(\tau(x)^{-1}x,\tau(x))$ is $\salg / \salg_{| \bS}
    \otimes \cB(\bR_{++})$-measurable.
  \end{itemize}
\end{lemma}

\begin{proof}
  Part (i) is clear from the definition of $\tau$.  Turning to part
  (ii), it suffices to show that the map $\tau$ is $\salg$-measurable,
  but this follows from Lemma~\ref{lemma:projection} and the fact that
  $\{x \in \bX : \tau(x) \ge t\} = [t,\infty) \bS$.  For part (iii),
  first note that if $\bX$ is a Borel space, then so is $\bS$ equipped
  with the trace $\sigma$-algebra and hence also $\bS \times
  \bR_{++}$.  Now apply the result of Kuratowski (see,
	for example, \cite[Section I.3]{par67}) that a measurable bijection between
  two Borel spaces has a measurable inverse.
\end{proof}

Proposition~\ref{proposition:renorm} required
Assumption~\ref{assumption:U} concerning the existence of suitable
countable family of subsets of the carrier space $\bX$ and the
hypothesis \eqref{Eq:nu-finite} that these sets all have finite $\nu$
mass.  The following result replaces such requirements by the
finiteness of integrals $\int h \, d\nu$ for a countable family of
measurable functions $h:\bX\mapsto\bR_+$, $n\in\bN$.

\begin{theorem}
  \label{thr:general-polar}
  Suppose that Assumption~\ref{assumption:X} holds.  Let $\cH$ be a
  countable family of measurable functions $h:\bX\mapsto\bR_+$ such
  that for all $h \in \cH$ and $x \in \bX$ the function $ t \mapsto
  h(tx)$ is right-continuous (or left-continuous) on $\bR_{++}$.

  A $\sigma$-finite measure $\nu$ on $\bX$ is $\alpha$-homogeneous and
  satisfies
  \begin{equation}
    \label{eq:nu-h-integral}
    \int_\bX h(x) \, \nu(dx) <\infty,\qquad h\in\cH,
  \end{equation}
  and 
  \begin{equation}
    \label{Eq:nu-support-h}
    \nu\left(\bigcap_{h\in\cH} \{x\in\bX:\; h(x)=0\}\right)=0
  \end{equation}
  if and only if there exists a probability measure $\pi$ on $\bX$
  such that \eqref{Eq:nu-rep} holds, 
  \begin{equation}
    \label{Eq:levy-condition-pi-general}
    \int_{\bX}\int_0^\infty h(tx) t^{-(\alpha+1)} \, dt \, \pi(dx) <\infty,
    \qquad h\in\cH,
  \end{equation}
  and 
  \begin{displaymath}
    \pi\left((0,\infty) \bigcap_{h \in \cH} \{x\in\bX:\; h(x)=0\}\right)=0. 
  \end{displaymath}
\end{theorem}

\begin{proof}
  \textsl{Necessity.} Enumerate $\cH$ as $\{h_n\}_{n \in \bN}$. Assume
  that for all $x \in \bX$ the function $t \mapsto h_n(tx)$ is 
	right-continuous in $t\in\bR_{++}$.  Denote by $\bQ_{++}$ the set of
  strictly positive rational numbers. For $n,k,j\in\bN$ and
  $r\in\bQ_{++}$, define
  \begin{displaymath}
    U_{nkrj}:=\{x\in\bS:\; h_n(sx)\geq 2^{-k},\; s\in[r,r+2^{-j}]\}.
  \end{displaymath}
  Since $B:=\{x \in \bS:\;h_n(x)\geq 2^{-k}\} \in \salg$, the
	right-continuity property and Remark~\ref{remark:scaling-one-s} yield that
  \begin{displaymath}
    U_{nkrj}=\bigcap_{s\in\bQ\cap[r,r+2^{-j}]}s^{-1}B\in \salg.
  \end{displaymath}

  Put $J_{rj} :=[r,r+2^{-j}]$. Then $J_{nj} U_{nkrj} \in \salg^\nu$
  and, by \eqref{eq:nu-h-integral},
  \begin{displaymath}
    \nu(J_{nj} U_{nkrj}) 
    \leq 2^k \int_{J_{nj} U_{nkrj}} h_n(x) \, \nu(dx)
    \leq 2^k \int_\bX h_n(x) \, \nu(dx)<\infty.
  \end{displaymath}
  
  Put $\bS' = \bigcup_{n,k,r,j} U_{nkrj}$.  If $x\in\bS\setminus\bS'$,
  $n,k,j \in \bN$, and $r \in \bQ_{++}$, there exists $s \in
  [r,r+2^{-j}]$ (depending on $x,n,k,r,j$) such that $h_n(s
  x)<2^{-k}$.  The right-continuity of $t \mapsto h_n(t x)$ yields
  that $h_n(r x)=0$ for all $r\in\bQ_{++}$ and thence $h_n(tx)=0$ for
  all $t\in\bR_{++}$.
	
  Enumerate $(U_{nkrj}, J_{rj})$, $n,k,j\in\bN$ and $r\in\bQ_{++}$, as
  $(W_n, I_n)$, $n \in \bN$, so that $\bS' = \bigcup_{n \in \bN} W_n$
  and $\nu(I_n W_n)< \infty$ for all $n \in \bN$.  By
  \eqref{Eq:nu-support-h}, the measure $\nu$ assigns all of its mass
  to the set $\bX':= (0,\infty) \bS' \in \salg^\nu$, and the result
  follows from the decomposition of $\nu$ guaranteed by
  Proposition~\ref{proposition:renorm} and
  Proposition~\ref{proposition:sphere}.  Finally,
  \eqref{Eq:levy-condition-pi-general} follows from the change of
  variables in \eqref{eq:nu-h-integral} using the polar decomposition of
  $\nu$ as $\pi\otimes\theta_\alpha$.

  If all functions $h_n(tx)$ are left continuous in $t\in\bR_{++}$,
  then the definition of $U_{nkrj}$ should be modified by working with
  the interval $[r-2^{-j},r]$ for $2^{-j}<r$. 

  \textsl{Sufficiency} is immediate by checking that the measure $\nu$
  defined by \eqref{Eq:nu-rep} is $\alpha$-homogeneous.
\end{proof}

\begin{remark}
  \label{rem:sphere-from-h}
  If $h_n(tx)\to 0$ as $t\downarrow0$ for all $n\in\bN$, then it is
  not necessary to require Assumption~\ref{assumption:X} in
  Theorem~\ref{thr:general-polar}. A measurable transversal can be
  constructed as follows. For each $n\in\bN$, let
  \begin{displaymath}
    W_n:=\{x\in\bX:\; h_k(x)=0,k<n,\; h_n(x)\neq 0\}.
  \end{displaymath}
Next, partition $W_n$ into measurable sets
  \begin{align*}
    W_{nj}&:=\{x\in W_n:\; \sup_{t>0}
    h_n(tx)\in(2^{-j},2^{-j+1}]\},\qquad j\geq 1,\\
    W_{n0}&:=\{x\in W_n:\; \sup_{t>0}
    h_n(tx)>1\}.
  \end{align*}
  Finally, define
  \begin{displaymath}
    \tau_{nj}(x):=\inf\{t>0:\; h_n(tx)>2^{-j}\},\qquad x\in W_{nj},
  \end{displaymath}
  and
  \begin{displaymath}
    S_{nj}:=\{\tau_{nj}(x)^{-1}x:\; x\in W_{nj}\}. 
  \end{displaymath}
  Then $\bS:=\bigcup_{n,j\in\bN} S_{nj}$ satisfies
  Assumption~\ref{assumption:X} in the complement of the set
  $\{x\in\bX:\; h_n(x)=0,n\in\bN\}$.
\end{remark}

\begin{remark}
  Theorem~\ref{thr:general-polar} asserts that an $\alpha$-homogeneous
  measure $\nu$ is the push-forward of the measure $\pi\otimes
  \theta_\alpha$ under the map $(x,t)\mapsto tx$ from $\bX \times
  \bR_{++}$ to $\bX$. In this case we say that $\nu$ admits a
  \emph{polar representation}.  It follows from the proof that we may
  replace $\bX$ by a subset that is invariant under the action of
  $\bR_{++}$ in such a way that the probability measure $\pi$ assigns
  all of its mass to a transversal.
\end{remark}

\section{Strictly stable random elements on convex cones}
\label{sec:strictly-stable-rand}
\label{sec:convex-cones}

\begin{definition}
  A \emph{convex cone} $\bK$ is an abelian measurable
  semigroup with neutral element $\neutral$ and a scaling
  $(x,a) \mapsto a x$ by $\bR_{++}$ that has the properties
  \begin{align*}
    a(x+y)&=ax+ay,\qquad a>0,\;x,y\in\bK,\\
    a\neutral&=\neutral,\qquad a>0.
  \end{align*}
\end{definition}

\begin{remark}
  The simplest examples of convex cones are $\bK = \bR^d$ and $\bK =
  \bR_+^d$ with the usual scaling by $\bR_{++}$, but there are many
  other examples, some of which we will recall later in this paper.
\end{remark}

\begin{remark}
  In contrast to the many classical studies of convex cones that are convex
  subsets of vector spaces over the reals which are closed under
  multiplication by nonnegative scalars (see, for example, \cite{fuc:lus81}),
  we do not assume the validity of the second distributivity law; that
  is, we do not require that $ax+bx=(a+b)x$ for $a,b > 0$ and $x \in
  \bK$.

  Stable random elements of convex cones have been studied in
  \cite{dav:mol:zuy08} under the assumptions that the scaling is
  jointly continuous and that $\bK':=\bK\setminus\{\neutral\}$ is a
  Polish space.
\end{remark}

\begin{definition}
  An \emph{involution} is a measurable map $x\mapsto x^*$ satisfying
  $(x+y)^*=x^*+y^*$, $(ax)^*=ax^*$, and $(x^*)^*=x$ for all
  $x,y\in\bK$ and $a>0$.  We assume that $\bK$ is equipped with an
  involution.
\end{definition}

\begin{definition}
  A measurable function $\chi$ that maps $\bK$ into the unit
  disk $\bD$ in the complex plane is called a {\em bounded
    semicharacter} (or, more briefly, a \emph{character}) if
  $\chi(\neutral)=1$, $\chi(x+y)=\chi(x)\chi(y)$, and
  $\chi(x^*)=\overline{\chi(x)}$ for all $x,y\in\bK$.
\end{definition}

\begin{remark}
  The family of all characters form a convex cone when equipped with
  pointwise multiplication as the semigroup operation, the topology of
  pointwise convergence, the neutral element $\one$ being the
  character identically equal to $1$, the involution being the complex
  conjugate, and the scaling defined by $(a\chi)(x):=\chi(ax)$,
  $x\in\bK$, $a>0$.
\end{remark}

We assume in the following that there exists a \emph{separating}
family $\bKH$ of characters in the usual sense that for each $x\neq y$
there exists $\chi\in\bKH$ such that $\chi(x)\neq\chi(y)$.  Such a
family does not exist for all semigroups, see
\cite[Ex.~8.20]{dav:mol:zuy08}. We also assume that the characters in
$\bKH$ generate the $\sigma$-algebra on $\bK$ and that the family
$\bKH$ is closed under taking finite products, contains the constant
function $1$, and so is a semigroup itself.  Note that $\bKH$ is not
assumed to be closed under scaling and so it is not necessarily a convex cone.

The distribution of a $\bK$-valued random element $\xi$ is
determined by its \emph{Laplace transform} $\bE\chi(\xi)$,
$\chi\in\bKH$, (see, for example, \cite[Th.~5.4]{dav:mol:zuy08}).
A random element $\xi$ is said to be \emph{symmetric} if
$\xi\eqd\xi^*$, that is, $\xi$ coincides in distribution with its
involution. The Laplace transform of a symmetric random element takes
only real values. Recall that the classical L\'evy-Khinchin-It\^o
description of infinitely divisible random elements of $\bR^d$ can
involve subtracting ``compensating'' terms to achieve convergence
of a sum of the points in a Poisson point process that would otherwise
be divergent, but that such compensation is not necessary when the
random elements are symmetric in the usual sense for $\bR^d$-valued
random elements (which is the special case of the sense considered
here with the involution given by $x \mapsto -x$).  Since no such
recentering using subtraction is possible in the general semigroup
setting, we mostly consider symmetric random elements. If the
involution is the identity, then all random elements are symmetric.

\begin{definition}
  A random element $\xi$ in $\bK$ is said to be \emph{strictly
    $\alpha$-stable} if
  \begin{equation}
    \label{Eq:stable}
    a^{1/\alpha}\xi'+b^{1/\alpha}\xi''
    \eqd (a+b)^{1/\alpha}\xi\,,\qquad a,b>0\,,
  \end{equation}
  where $\xi'$ and $\xi''$ are i.i.d. copies of $\xi$.
\end{definition}

In general cones, any value $\alpha\neq0$ is possible, see
\cite{dav:mol:zuy08}. Since by redefining the scaling from $ax$ to
$a^{-1}x$ it is possible to turn a negative $\alpha$ into a positive
one, in the following we consider only the case $\alpha>0$.

\begin{remark}
  An alternative definition of strictly stable random elements that
  coincides with the above for $\bR^d$ (and in many other situations)
  requires that, for all $n\geq2$, there exists $a_n>0$ such that
  $\xi_1+\cdots+\xi_n\eqd a_n\xi$, where $\xi,\xi_1,\dots,\xi_n$ are
  i.i.d. and $n\geq2$. While \eqref{Eq:stable} implies this condition
  immediately, extra assumptions related to the semicontinuity
  property of characters and the continuity of the scaling operation
  are needed for the equivalence of the two definitions, see
  \cite[Th.~5.16]{dav:mol:zuy08}. The major step is to establish that
  $a_n=n^{1/\alpha}$, after which \eqref{Eq:stable} follows easily.
\end{remark}

A strictly $\alpha$-stable $\xi$ is always \emph{infinitely divisible}
and so its Laplace transform satisfies
\begin{equation}
  \label{Eq:Laplace-exponent}
  \bE\chi(\xi)=\exp\{-\phi(\chi)\},\qquad \chi\in\bKH\,,
\end{equation}
for a negative definite complex-valued function $\phi$ on $\bKH$ with
$\Re\phi\in[0,\infty]$ and $\phi(\one)=0$, see \cite[Th.~3.2.2,
Prop.~4.3.1]{ber:c:r}. The function $\phi$ is called the \emph{Laplace
  exponent} of $\xi$. 
	
A Laplace exponent is associated with a unique \emph{L\'evy measure}
$\nu$, a Radon measure on the semigroup of all bounded characters on
$\bKH$, see \cite{ber:c:r}.  In the following we always assume that the
L\'evy measure is supported by $\bK':=\bK\setminus\{\neutral\}$
canonically embedded 
into the semigroup of all bounded semicharacters on $\bKH$ and
that it is $\sigma$-finite on $\bK'$. In this case we say that $\xi$
\emph{admits a L\'evy measure}.  The L\'evy measure $\nu$ satisfies
\begin{equation}
  \label{eq:lev-mes}
  \int_{\bK'}(1-\Re \chi(x))\,\nu(dx)<\infty
\end{equation}
for all $\chi\in\bKH$.  The measure $\nu$ is the intensity measure of
a Poisson process $\{\eta_i: i\in\bN\}$ on $\bK'$, and the
appropriately defined (if necessary using principal values and
compensating terms) sum of the points $\eta_i$ yields an infinitely
divisible random element that is said to not have deterministic,
Gaussian or idempotent components.

\begin{lemma}
  \label{lemma:non-zero}
  Assume for some $\chi\in\bKH$ that
  \begin{equation}
    \label{Eq:lim-chi-positive}
    \liminf_{t\downarrow0} \Re\chi(tx)>0
  \end{equation}
  for all $x\in\bK$.  If $\xi$ is a strictly $\alpha$-stable random
  element, then $\bE\chi(\xi)\neq 0$.
\end{lemma}

\begin{proof}
  Since $-1 \le \Re\chi(t x) \le 1$ for all $t > 0$ and $x \in \bK$,
  it follows from the assumption \eqref{Eq:lim-chi-positive} and
  Fatou's Lemma that
  \begin{equation}
    \label{eq:Fatou_for_char}
    0 < \bE \left[\liminf_{t\downarrow0} \Re\chi(t\xi)\right]
    \le \liminf_{t\downarrow0} \bE \Re\chi(t\xi).
  \end{equation}
  By the stability property of $\xi$, $\bE\chi(\xi) =
  (\bE\chi(n^{-1/\alpha}\xi))^n$ for all $n \in \bN$, and so
  $\bE\chi(\xi) = 0$ would imply that $\bE\chi(n^{-1/\alpha}\xi)=0$
  and hence $\bE\,\Re\chi(n^{-1/\alpha}\xi)=0$ for all $n \in \bN$,
  but this contradicts \eqref{eq:Fatou_for_char}.
\end{proof}

\begin{lemma}
  \label{lemma:homogeneous}
  Assume that \eqref{Eq:lim-chi-positive} holds for all $\chi\in\bKH$.
  Then the L\'evy measure of a strictly $\alpha$-stable random
  element is an $\alpha$-homogeneous measure on $\bK'$.
\end{lemma}
\begin{proof}
  It follows from \eqref{Eq:stable} that 
  \begin{displaymath}
    \phi(a^{1/\alpha}\chi)+\phi(b^{1/\alpha}\chi)
    =\phi((a+b)^{1/\alpha}\chi),\quad a,b>0\,,
  \end{displaymath}
  where the Laplace exponent $\phi$ of the strictly stable random
  element $\xi$ is finite by Lemma~\ref{lemma:non-zero}.  Since
  $(x,a)\mapsto ax$ is a jointly measurable map, the function $a
  \mapsto \bE\chi(a\xi)$ is measurable by Fubini's theorem. Therefore,
  \begin{equation}
    \label{eq:phi-homogeneous}
    \phi(a\chi)=a^\alpha\phi(\chi),\qquad a>0.
  \end{equation}
  The random element $a\xi$ is also infinitely divisible and its
  L\'evy measure is $B \mapsto \nu(aB)$ for measurable subsets $B\subseteq\bK'$. By
  \eqref{eq:phi-homogeneous}, the L\'evy measure of $a\xi$ is
  $a^\alpha\nu$. Since the L\'evy measure is unique, we obtain
  \eqref{eq:nu-x-homogeneous}.
\end{proof}

\begin{theorem}
  \label{thr:levy-k-separating-countable}
  Let $\xi$ be a strictly stable random element that admits a
  $\sigma$-finite L\'evy measure $\nu$.  Assume that there is a
  countable, closed under finite products, separating family of
  characters $\bKH$ such that $t \mapsto \Re\chi(tx)$, $t\in\bR_{++}$,
  is right- (or left-) continuous for all $x\in\bK'$ and $\chi \in
  \bKH$. Assume also that Assumption~\ref{assumption:X} holds and
  \eqref{Eq:lim-chi-positive} holds for all $\chi\in\bKH$.

  Then $\nu$ admits the polar representation
  $\pi\otimes\theta_\alpha$, where $\pi$ is a probability measure on
  $\bK'$ satisfying
  \begin{equation}
    \label{eq:eps-condition}
    \int_0^\infty \bE[1-\Re\chi(t\eps)] t^{-(\alpha+1)} \, dt <\infty,
    \qquad \chi\in\bKH.
  \end{equation}
  The Poisson process with intensity measure $\nu$ can be represented
  as $\{\Gamma_i^{-1/\alpha}\eps_i\}_{i\in\bN}$, where
  $\{\eps_i\}_{i\in\bN}$ is a sequence of i.i.d. copies of a random
  element $\eps$ in $\bK'$ with distribution $\pi$, and
  $\{\Gamma_i\}_{i\in\bN}$ are successive points of a unit intensity
  Poisson process on $\bR_+$.  If $\xi$ is symmetric, then $\eps$ can
  also be chosen to have a symmetric distribution.
\end{theorem}
\begin{proof}
  The measure $\nu$ admits the polar decomposition
  $\pi\otimes\theta_\alpha$ by Theorem~\ref{thr:general-polar} applied
  with $\bX:=\bK'$, $\cH := \{1 - \Re \chi: \; \chi \in \bKH\}$, and
  $\nu$ being the L\'evy measure of $\xi$. Indeed,
  \eqref{eq:nu-h-integral} follows from \eqref{eq:lev-mes}, $\nu$ is
  $\alpha$-homogeneous by Lemma~\ref{lemma:homogeneous} given that
  \eqref{Eq:lim-chi-positive} is assumed to hold, and the separating
  condition imposed on $\bKH$ yields \eqref{Eq:nu-support-h}.

  The Poisson point process with intensity measure $\theta_\alpha$ is
  obtained as $\{\Gamma_i^{-1/\alpha}: i\in\bN\}$. Thus, the Poisson
  process on $\bK'\times \bR_{++}$ with intensity measure
  $\pi\otimes\theta_\alpha$ is obtained by marking a Poisson point
  process $\{\Gamma_i^{-1/\alpha}: i\in\bN\}$ with marks
  $\{\eps_i\}_{i\in\bN}$ that are i.i.d. copies of a random element
  $\eps$ in $\bK'$ with distribution $\pi$. By
  \eqref{Eq:levy-condition-pi-general}, the random element $\eps$
  satisfies \eqref{eq:eps-condition}. Since $\nu$ is the push-forward
  of $\pi\otimes\theta_\alpha$ under the multiplication map
  $(x,t)\mapsto tx$, the Poisson process with intensity $\nu$ is given
  by $\{\Gamma_i^{-1/\alpha}\eps_i:\; i\in\bN\}$.
  
  The uniqueness of the L\'evy measure yields that $\nu$ is symmetric
  if $\xi$ is symmetric. Then in the proof of
  Proposition~\ref{proposition:renorm} it suffices to replace each set
  $W_k$ with the union of $W_k$ and its image under the involution to
  ensure that $\pi$ is a symmetric measure.
\end{proof}

\begin{remark}
  Note that the random element $\eps$ in
  Theorem~\ref{thr:levy-k-separating-countable} is not unique and also
  is not restricted to belong to the transversal $\bS$ from
  Assumption~\ref{assumption:X}. By rescaling it is always possible to
  arrange that $\eps\in\bS$ a.s., however in this case, the Poisson
  process with intensity $\nu$ is given by
  $\{c\Gamma_i^{-1/\alpha}\eps_i:\; i\in\bN\}$ for a positive scaling
  constant $c$.
\end{remark}

\begin{remark}
  Suppose that the ``L\'evy-Khinchin-It\^o'' decomposition of the strictly
  $\alpha$-stable random element $\xi$ does not contain any deterministic, Gaussian
  or idempotent components, so that $\xi$ is the sum of the points in
  a Poisson process $\{\eta_i : i \in \bN\}$, where the sum is
  appropriately defined by using principal values and compensating
  terms if necessary, see \cite[Th.~7.2]{dav:mol:zuy08}.
  Theorem~\ref{thr:levy-k-separating-countable} establishes that
  $\{\eta_i : i \in \bN\}$ can be constructed as
  $\{\Gamma_i^{-1/\alpha}\eps_i: i \in \bN\}$, that is, as randomly
  scaled i.i.d. copies of a random element $\eps$.

  Recall that no compensating terms for the sum of the $\eta_i$ are
  required if $\xi$ is $\alpha$-stable and symmetric. In this case,
  the Laplace exponent of $\xi$ is given by
  \begin{displaymath}
    \phi(\chi) =\int_{\bK'}(1-\chi(x))\,\nu(dx)
    =\int_{\bK'}(1-\Re\chi(x))\,\nu(dx)\,,
    \quad \chi\in\bKH\,.
  \end{displaymath}
  Put
  \begin{equation}
    \label{eq:xi-r}
    \xi^{(r)} :=\sum_{\Gamma_i\leq r}
    \Gamma_i^{-1/\alpha}\eps_i\,,\qquad r>0.
  \end{equation}
  The probability generating functional formula for a Poisson point
  process yields that, for all $\chi\in\bKH$,
  \begin{align*}
    \bE \chi(\xi^{(r)}) 
    & =\exp\left\{-\int_{r^{-\alpha}}^\infty \int_{\bK'}
      (1-\chi(tx))\alpha t^{-(\alpha+1)} \,  \pi(dx) \, dt\right\}\\
    & \to \bE \chi(\xi)\qquad \text{as }
    r\uparrow\infty, 
  \end{align*}
  since $\bE\chi(t\eps)$ is real-valued by the symmetry of $\eps$ and
  the integral under the exponential converges by
  \eqref{eq:eps-condition}. 
\end{remark}

\section{Examples}
\label{sec:applications}

\begin{example}
  \label{example:Rd}
  If $\bK=\bR^d$ with the usual arithmetic addition, the involution given by
  the negation, the Borel $\sigma$-algebra, and conventional scaling
  by positive scalars, then Assumption~\ref{assumption:X} holds with
  $\bS$ being the sphere in $\bR^d$ with respect to any norm. A
  countable separating family of continuous characters is given by
  $\chi(x)=\exp\{\imath \langle u,x\rangle\}$, where $\langle
  u,x\rangle$ is the scalar product of $u\in\bQ^d$ and $x\in\bR^d$,
  and Theorem~\ref{thr:levy-k-separating-countable} yields the polar
  representation of L\'evy measures for strictly stable random
  vectors. In particular, each strictly stable random vector $\xi$
  corresponds to the Poisson process
  $\{\Gamma_i^{-1/\alpha}\eps_i:i\in\bN\}$. If $\alpha\in(0,1)$ or if
  $\xi$ is symmetric, then the sum of these points converges almost
  surely and yields the LePage series decomposition of $\xi$, see
  \cite{lep:wood:zin81}.
\end{example}

\begin{example}
  \label{ex:operator}
  In the setting of Example~\ref{example:Rd}, define the scaling by
  letting $tx:=\exp\{(\log t)\AA\}x$, $t>0$, for a non-degenerate
  matrix $\AA$. The corresponding stable elements are usually called
  \emph{operator stable}, see \cite{hud:mas81,shar69}. An operator
  stable random element is infinitely divisible and so admits a
  series representation and a L\'evy measure. By
  Theorem~\ref{thr:levy-k-separating-countable}, an operator stable
  random element $\xi$ in $\bR^d$ admits the LePage series
  representation
  \begin{displaymath}
    \xi\eqd \sum_{i\in\bN} \exp\{-\alpha^{-1}(\log\Gamma_i)\AA\} \eps_i. 
  \end{displaymath}
  Note that the polar decomposition of the L\'evy measure appears in
  \cite[Eq.~(2)]{hud:jur:veeh86} and \cite{hud:mas81}.
\end{example}

\begin{example}
  \label{ex:processes-values}
  Let $\bK$ be the family $\bR^{[0,\infty)}$ of real valued-functions
  on $\bR_+$ with the cylindrical $\sigma$-algebra, the pointwise arithmetic
  addition of functions as the semigroup operation, involution being
  the negation, and the scaling operation applied pointwise.
	It is known \cite{kab:stoev12,mar70} that an infinitely
  divisible separable in probability stochastic process $\xi(s)$,
  $s\in\bR_+$, can be associated with the Poisson process
  $\{\eta_i,i\in\bN\}$ on $\bR^{[0,\infty)}$ with a $\sigma$-finite
  intensity measure $\nu$, so that $\xi$ admits a L\'evy measure
  $\nu$. If $\xi$ is symmetric, then, for each $s\in\bR_+$,
  \begin{equation}
    \label{eq:id-processes}
    \xi(s)\eqd \lim_{r\downarrow 0}\sum_{|\eta_i(t)|\geq r}
    \eta_i(s),\qquad s\in\bR_+,
  \end{equation}
  where $\eqd$ in this setting denotes the equality of all
  finite-dimensional distributions.  It should be noted that the order
  of summands in \eqref{eq:id-processes} may change with $t$, and
  there might be no order of the summands that guarantees the
  convergence for all rational $s$, not to say for all
  $s\in\bR_+$. Such a common order exists for random functions with
  non-negative values (in which case $\bK$ is endowed with the
  identical involution).
  
  This cone does not admit a countable separating family of
  characters.  Let $\bKD$ be the countable family of characters
  defined by
  \begin{equation}
    \label{Eq:chi-line}
    \chi(x)=e^{\imath x(s) u},\qquad s\in\bQ_+,
  \end{equation}
  where $x=(x(s))_{s\geq 0}$ is an element of $\bR^{[0,\infty)}$ and
  $u$ belongs to the set $\bQ$ of rational numbers.  
  Theorem~\ref{thr:levy-k-separating-countable} applies to the process
  $\xi$ restricted onto the set $\bQ_+$ of non-negative rationals, so
  that the image $\tilde{\nu}$ of the L\'evy measure $\nu$ under the
  map that restricts a function to $\bQ_+$ has polar decomposition as
  $\pi\otimes\theta_\alpha$ and the Poisson process with intensity
  $\tilde{\nu}$ is given by
  $\{\Gamma_i^{-1/\alpha}\tilde{\eps}_i,i\geq1\}$. By
  \cite[Prop.~2.19]{kab:stoev12}, for each $t\in\bR_+\setminus\bQ_+$,
  the separability property of $\xi$ yields that
  $\Gamma_i^{-1/\alpha}\tilde{\eps}_i(t_n)$ converges in probability
  to $\eta_i(t)$, whence $\tilde{\eps}_i(t_n)$ converges in
  probability to $\eps_i(t):=\Gamma_i^{1/\alpha}\eta_i(t)$. Thus, the
  L\'evy measure $\nu$ corresponds to the Poisson process
  $\{\Gamma_i^{-1/\alpha}\eps_i,i\geq1\}$, where $\{\eps_i,i\in\bN\}$
  is a sequence of i.i.d. separable processes distributed as
  $\eps(s)$, $s\in\bR_+$. Condition \eqref{eq:eps-condition} in this
  setting is equivalent to $\bE|\eps(s)|^\alpha<\infty$ for all
  $s$. 

  In the symmetric case, following \eqref{eq:xi-r}, the convergence in
  \eqref{eq:id-processes} can be rephrased as
  \begin{displaymath}
    \xi(s)\eqd \lim_{r\uparrow\infty} \sum_{\Gamma_i\leq r} 
    \Gamma_i^{-1/\alpha}\eps_i(s),\qquad s\in\bR_+.
  \end{displaymath}
  Therefore, in case of symmetric $\alpha$-stable processes the order
  of summands in \eqref{eq:id-processes} can be made the same for all
  time points. This makes it possible to appeal to path
  regularity results for stochastic integrals from
  \cite[Th.~4]{ros89p} in order to confirm that $\eps$ shares path
  regularity properties with $\xi$, for instance, $\eps$ is almost
  surely right-continuous with left limits (c\`adl\`ag) if $\xi$ is
  c\`adl\`ag. The same holds for stochastic processes with almost
  surely non-negative values, since then the involution is the
  identity. The pointwise convergence of the LePage series yields the
  uniform convergence, see \cite{bas:ros13}. Note that a result
  concerning the existence of the series representation of a general
  (not necessarily symmetric stable) infinitely divisible c\`adl\`ag
  function using c\`adl\`ag summands is not available.
\end{example}

\begin{example}
  Let $\bK$ be the family of non-negative functions $x(s)$,
  $s\in\bR_+$, with the cylindrical $\sigma$-algebra, the semigroup
  operation being pointwise maximum, identical involution, and the
  scaling applied poinwisely to the values of the function. It is
  shown in \cite{kab:stoev12} that each separable max-infinitely
  divisible stochastic process admits a L\'evy measure. A separating
  family of characters is given by those of the form $x \mapsto
  \chi(x) = \one_{x(s)< a}$ for $s,a\in\bR_+$. While these characters
  are not continuous, the function $t\mapsto \chi(tx)$ is
  right-continuous. Restricting the functions to non-negative
  rationals as in Example~\ref{ex:processes-values} and using the
  results from \cite{kab:stoev12} concerning the max-infinitely
  divisible setting, we obtain that the L\'evy measure of each
  max-stable separable in probability stochastic process $\xi$ admits
  a polar representation and the process itself admits the series
  representation
  \begin{displaymath}
    \xi(s)\eqd \bigvee_{i\in\bN}\Gamma_i^{-1/\alpha}\eps_i(s), 
  \end{displaymath}
  which first appeared in \cite{haan84}.
\end{example}

\begin{example}
  \label{ex:time-stable}
  Equip the family $\bK$ of real-valued c\`adl\`ag functions on
  $\bR_+$ that vanish at the origin with pointwise arithmetic addition, involution
  being the negation, and the scaling defined by  $(tx)(s):=x(ts)$, $s\in\bR_+$, for
  $t\in\bR_{++}$. Stable elements in this cone with $\alpha=1$ are
  called \emph{time-stable} processes in \cite{kop:mol15} and
  {\em processes infinitely divisible with respect to time} in \cite{man05}.
  The characters are given by \eqref{Eq:chi-line} and, in view of the
  c\`adl\`ag assumption, they constitute a countable separating family
  and are right-continuous as required in
  Theorem~\ref{thr:levy-k-separating-countable}. Furthermore,
  $\chi(tx)=\exp\{\imath x(ts)u\}\to 1$ as $t\downarrow0$, so that a
  transversal in $\bK$ can be constructed as in
  Remark~\ref{rem:sphere-from-h}, see also \cite{kop:mol15}. Thus, if
  a time-stable process with symmetric distribution admits a series
  representation with c\`adl\`ag functions, it also admits the LePage
  representation
  \begin{displaymath}
    \xi(s)\eqd \sum_{i\in\bN} \eps_i(\Gamma_i^{-1}s),\qquad s\in\bR_+,
  \end{displaymath}
  where $\{\eps_i\}_{i\in\bN}$ are i.i.d. copies of $\eps(s)$,
  $s\in\bR_+$, such that 
  \begin{displaymath}
    \bE \int_0^\infty \min(1,\eps(s)^2)s^{-2}ds<\infty
  \end{displaymath}
  because of \eqref{eq:eps-condition}. 

  The setting can be altered by considering the family of non-negative
  c\`adl\`ag functions with the identical involution. In particular,
  each seperable in probability c\`adl\`ag time-stable
  process with non-negative values admits a LePage series
  representation.
\end{example}

\begin{example}
  Let $\bK$ be the family of locally finite measures $\mu$ on $\bR^d$
  with the arithmetic addition, identical involution, and the Borel
  $\sigma$-algebra generated by the vague topology
  \cite[Sec.~9.1]{MR2371524}. An infinitely divisible locally finite
  random measure admits a L\'evy measure, see \cite{MR2371524}. A
  countable separating family of continuous characters consists of
  \begin{displaymath}
    \chi(\mu)=\exp\left\{-\int ud\mu\right\},
  \end{displaymath}
  where $u$ belongs to an appropriately chosen countable family of
  continuous functions with compact support. If the scaling operation
  is applied to the values of measures, then $\chi(tx)\to1$ as
  $t\downarrow0$ for all $x$, so that a transversal can be constructed
  as in Remark~\ref{rem:sphere-from-h}. An alternative way of
  constructing a measurable trasversal $\bS$ is to take a sequence
  $\{B_k\}_{k\in\bN}$ of bounded sets that form a base of the topology
  and let $\mu\in\bS$ if $\mu(B_0)=\cdots=\mu(B_{n-1})=0$ and
  $\mu(B_n)=1$ for some $n$.  By
  Theorem~\ref{thr:levy-k-separating-countable}, each stable locally
  finite random measure admits the LePage representation
  \begin{displaymath}
    \mu\eqd \sum_{i\in\bN} \Gamma_i^{-1/\alpha}\eps_i 
  \end{displaymath}
  for a sequence $\{\eps_i\}_{i\in\bN}$ of i.i.d. locally finite
  measures with $\alpha$-integrable values.
\end{example}

\begin{example}
  Let $\bK$ be the family of closed sets $F$ in $\bR^d\setminus\{0\}$
  with the union as the semigroup operation, identical involution, the
  conventional scaling, and the $\sigma$-algebra generated by families
  $\{F:F\cap K\neq\emptyset\}$ for all compact sets $K$. The countable
  separating family of continuous characters is given by
  $\chi(F):=\one_{F\cap G=\emptyset}$ for open sets $G$ from
  the base of topology on $\bR^d$. Note that deterministic closed sets
  have idempotent distributions in this cone. It is known that each
  union infinitely divisible random closed set admits a L\'evy
  measure, see \cite{MR2132405}. Thus, each strictly $\alpha$-stable
  random closed set $\xi$ without idempotent factors (so that
  $\bP\{x\in \xi\}<1$ for all $x\in\bR^d$) admits the series
  representation as the union of $\Gamma_i^{-1/\alpha}\eps_i$,
  $i\in\bN$, for a sequence $\{\eps_i\}_{i\in\bN}$ of i.i.d.  random
  closed sets.
\end{example}

\begin{example}
  Let $\bK$ be the family of all non-decreasing functions
  $\Phi:\cK\mapsto\bR_+$ defined on the family $\cK$ of compact
  subsets of $\bR^d$ that vanish at the empty set and are upper
  semicontinuous, that is $\Phi(K_n)\downarrow\Phi(K)$ as
  $K_n\downarrow K$. Such functions are known as capacities (and
  sometimes are called topological pre-capacities, see
  \cite{MR1814344}). Equip $\bK$ with the semigroup operation by
  taking pointwise maximum of two capacities and the scaling of their
  values. It is shown in \cite{nor86} that an infinitely divisible
  capacity $\xi$ admits a L\'evy measure. It does not have idempotent
  factors if the essential infimum of $\xi(K)$ vanishes for all
  $K\in\cK$. By Theorem~\ref{thr:levy-k-separating-countable}, each
  strictly $\alpha$-stable capacity admits the series representation
  $\sum_{i \in \bN} \Gamma^{-1/\alpha}\eps_i$, see also
  \cite[Th.~4.1]{mol:strok15}.
\end{example}

\begin{example}
  Let $\bK$ be the family of metric measure spaces with the Cartesian
  product as the semigroup operation and the scaling applied to the
  metric, see \cite{evan:mol15} for details. This semigroup admits a
  countable separating family of characters and a measurable
  transversal. Furthermore, each infinitely divisible random element
  in $\bK$ admits a L\'evy measure and so
  Theorem~\ref{thr:levy-k-separating-countable} applies and yields the
  LePage series representation of stable metric measure spaces
  obtained in \cite[Th.~10.3]{evan:mol15}.
\end{example}

\section*{Acknowledgment}
\label{sec:acknowledgment}

The paper was initiated while SE was visiting the University of Bern
supported by the Swiss National Science Foundation.

%\bibliographystyle{abbrv} 
%\bibliographystyle{amsalpha} 
%\bibliography{se-im}

\def\cprime{$'$}

\end{document}